\documentclass[a4paper, 12pt]{article}

\usepackage{amsmath, amsthm, amssymb}
\usepackage{graphicx}
\usepackage{enumitem}
\usepackage{authblk}
\usepackage[textsize=small, textwidth=2cm, color=yellow]{todonotes}
\usepackage{thmtools}
\usepackage{thm-restate}

\usepackage[utf8]{inputenc}
\usepackage[russian, english]{babel}
\usepackage[T1]{fontenc}
\usepackage[a4paper, margin=2.5cm]{geometry}
\usepackage{needspace}

\usepackage{libertine} 
\usepackage{inconsolata} 

\usepackage[hyphens]{url}
\usepackage[linktoc = all, hidelinks, colorlinks, unicode=true, pagebackref=true]{hyperref}
\usepackage[capitalise, compress, nameinlink, noabbrev]{cleveref}
\hypersetup{linkcolor={black}, citecolor={blue!70!black}, urlcolor={blue!70!black}}

\declaretheorem[name=Theorem]{theorem}

\def\cqedsymbol{\ifmmode$\lrcorner$\else{\unskip\nobreak\hfil
\penalty50\hskip1em\null\nobreak\hfil$\lrcorner$
\parfillskip=0pt\finalhyphendemerits=0\endgraf}\fi} 

\let\le\leqslant
\let\ge\geqslant
\let\leq\leqslant
\let\geq\geqslant

\makeatletter
\def\@setthanks{\vspace{-\baselineskip}\def\thanks##1{\@par\noindent##1\@addpunct.}\thankses}
\def\@setaddresses{%
  \par\nobreak\begingroup\footnotesize\interlinepenalty\@M\bigskip
  \def\author##1{}%
  \def\address##1##2{%
    \par\addvspace\medskipamount\noindent
    \@ifnotempty{##1}{\textit{##1}: }{\ignorespaces##2}%
  }%
  \def\email##1##2{\ifvmode\else\unskip; \fi\textit{email}: ##2}%
  \addresses\par
  \endgroup
}

\title{Surfaces without quasi-isometric \\ simplicial triangulations}

\author{James Davies\thanks{
The author was supported by the Alexander von Humboldt Foundation in the framework of the Alexander von Humboldt Professorship of Daniel Kráľ endowed by the Federal Ministry of Education and Research.}}

\affil{Leipzig University, Germany}

\date{}

\begin{document}

\maketitle

\begin{abstract}
    We construct a complete Riemannian surface $\Sigma$ that admits no triangulation $G\subset \Sigma$ such that the inclusion $G^{(1)} \hookrightarrow \Sigma$ is a quasi-isometry, where $G^{(1)}$ is the simplicial 1-skeleton of $G$.
    Our construction is without boundary, has arbitrarily large systole, and furthermore, there is no embedded graph $G\subset\Sigma$ such that $G^{(1)} \hookrightarrow \Sigma$ is a quasi-isometry.
    This answers a question of Georgakopoulos.
\end{abstract}

\section{Introduction}\label{sec:intro}

It is well known that for every complete Riemannian surface $\Sigma$, there is a locally finite metric graph $G\subset \Sigma$ whose metric is induced by $\Sigma$ and so that the inclusion map $G \hookrightarrow \Sigma$ is a quasi-isometry~\cite{bonamy2023asymptotic,burago2001course,creutz2022triangulating,georgakopoulos2026triangulating,ntalampekos2023polyhedral,saucan2008intrinsic}. In fact, Ntalampekos and Romney~\cite{ntalampekos2023polyhedral} showed that $G\subset \Sigma$ can even be chosen so that $G \hookrightarrow \Sigma$ is a $(1,\epsilon)$-quasi-isometry for any $\epsilon > 0$ and
Georgakopoulos and Vigolo \cite{georgakopoulos2026triangulating} recently showed that $G\subset \Sigma$ can furthermore also be taken to be a triangulation of $\Sigma$.
In other words, the coarse geometry of Riemannian surfaces can be approximated arbitrarily well by metric graphs triangulating the surface.

Such results or similar have seen a number of applications.
For instance, Bonamy, Bousquet, Esperet, Groenland, Liu,
Pirot, and Scott \cite{bonamy2023asymptotic} used this to prove that complete Riemannian surfaces of bounded genus have asymptotic dimension at most 2. 
Maillot \cite{maillot2001quasi} used a similar result in his proof that virtual surface groups are exactly the groups quasi-isometric to a complete simply-connected Riemannian surfaces.
Ntalampekos and Romney \cite{ntalampekos2023polyhedral} used a version for length surfaces as a step in their proof that any length surface is the Gromov–Hausdorff limit of polyhedral surfaces with controlled geometry. They gave further applications of this including a new proof of the Bonk-Kleiner theorem \cite{bonk2002quasisymmetric} characterizing Ahlfors 2-regular quasispheres.

The coarse geometry of metric graphs tends to be much more complicated and harder to work with than that of simplicial graphs (where all edge lengths are equal to 1).
For instance, one can compare the proofs that minor-free (simplicial) graphs \cite{bonamy2023asymptotic} and minor-free metric graphs \cite{liu2023assouad} have asymptotic dimension at most 2.
So, it is therefore desirable to determine when metric graphs with particular properties are quasi-isometric
to simplicial graphs with analogous properties.
Such a theorem has been proven for planar metric graphs \cite{davies2025string} and very recently for minor-free metric graphs~\cite{davies2025coarse}.

Georgakopoulos \cite{394166} asked if these results for Riemannian surfaces on quasi-isometric embedded metric graphs~\cite{bonamy2023asymptotic,burago2001course,creutz2022triangulating,georgakopoulos2026triangulating,ntalampekos2023polyhedral,saucan2008intrinsic} could be improved so that there is some embedded graph or even triangulation $G\subset \Sigma$ such that $G^{(1)} \hookrightarrow \Sigma$ is a quasi-isometry, where $G^{(1)}$ is the simplicial 1-skeleton of $G$.
Recently, Georgakopoulos and Vigolo \cite{georgakopoulos2026triangulating} proved that complete Riemannian surfaces with uniform nets admit such quasi-isometric simplicial triangulations, which improved on some similar results of Maillot \cite{maillot2001quasi}.
Answering a question of Georgakopoulos and Papasoglu~\cite{georgakopoulos2023graph},
the author recently proved that complete Riemannian surfaces of bounded genus have such simplicial triangulations \cite{davies2025string}.
Bowditch \cite{bowditch2020bilipschitz} also proved that complete smooth $n$-dimensional Riemannian manifolds of bounded geometry admit quasi-isometric simplicial triangulations.
In this paper we answer Georgakopoulos's \cite{394166} question in the negative.

\begin{theorem}\label{Rie:main}
    There is a complete Riemannian surface $\Sigma$ with no triangulation $G\subset \Sigma$ such that $G^{(1)} \hookrightarrow \Sigma$ is a quasi-isometry.
\end{theorem}

Our constructed surfaces are without boundary and have arbitrarily large systole (the minimum length of a non-contractible loop).
In other words, the surfaces can be taken to be arbitrarily locally planar. This is perhaps surprising as it contrasts with the fact that complete Riemannian planes do admit quasi-isometric simplicial triangulations \cite{davies2025string}.
Furthermore, our construction shows that there is no such embedded graph $G\subset \Sigma$, rather than just no such triangulation.
So, the strongest version of our result is as follows.

\begin{theorem}\label{Rie:main2}
    For every $K\ge 0$, there exists a complete Riemannian surface $\Sigma$ without boundary, with systole at least $K$, and with no embedded graph $G\subset \Sigma$ such that $G^{(1)} \hookrightarrow \Sigma$ is a quasi-isometry.
\end{theorem}

Since complete Riemannian surfaces do admit quasi-isometric triangulations by metric graphs \cite{georgakopoulos2026triangulating} (or even just have quasi-isometric embedded graphs \cite{bonamy2023asymptotic,burago2001course,creutz2022triangulating,georgakopoulos2026triangulating,ntalampekos2023polyhedral,saucan2008intrinsic}), these results are also another good demonstration of how the coarse structure of metric graphs tends to be more complex than that of simplicial graphs.

To prove \cref{Rie:main2}, we will actually prove a version that for any fixed $M,A$ gives compact Riemannian surfaces of bounded genus that have no such $(M,A)$-quasi-isometry (see \cref{Rie:main3}).
In \cref{sec:conclude} we discuss the resulting bounds on the genus of the surfaces constructed in \cref{Rie:main3} in comparison with the results of \cite{davies2025string} that bounded genus complete Riemannian surfaces do admit quasi-isometric simplicial triangulations.
In the next section, we prove \cref{Rie:main2}.

\section{Proof}\label{sec:proof}

Before proving \cref{Rie:main2}, we introduce some necessary preliminaries.

For~$M, A \ge 0$ with $M \geq 1$, an \emph{$(M,A)$-quasi-isometry} from a metric space~$X$ to another metric space~$Y$ is a map~$f: X \to Y$ such that
\begin{enumerate}
    \item $M^{-1} d_X(u, v) - A \leq d_Y(f(u),f(v)) \leq M  d_X(u,v) + A$ for every~$u, v \in X$, and
    \item for every $y\in Y$, there exists some $x\in X$ with $d_Y(y,f(x)) \leq A$.
\end{enumerate}
We say that $X$ is \emph{$(M,A)$-quasi-isometric} to $Y$ if there exists an $(M,A)$-quasi-isometry from $X$ to $Y$.
A \emph{quasi-isometry} is just a $(M,A)$-quasi-isometry for some $M,A\ge 0$ with $M\ge 1$, and we say that 
two metric spaces $X$ and $Y$ are \emph{quasi-isometric} if there is a quasi-isometry between $X$ and $Y$.

A \emph{Riemannian surface} is a surface $\Sigma$ together with a \emph{Riemannian metric} $d_\Sigma$ defined by a scalar product on the tangent space of every point.
A Riemannian surface $\Sigma$ is \emph{complete} if the metric space $(\Sigma, d_\Sigma)$ is complete.
Note that compact Riemannian surfaces are also complete.
The \emph{systole} of a Riemannian surface is equal to the infimum of the lengths of its non-contractible loops.
For more on Riemannian surfaces, see \cite{spivak1979comprehensive}.

In our construction we shall start by constructing some surface and then choose some suitable (Riemannian) metric on the surface.
We give a sketch of a standard way of constructing a Riemannian surface starting with a smooth surface $\Sigma$ and a locally finite triangulation of $H \subset \Sigma$ by a metric graph $H$ such that for each triangular face, the length of the longest side is less than the sum of the other two sides.
This will just be one way to see in the proof of \cref{Rie:main2} that the metric we choose can further be chosen to be a Riemannian metric.
Such smoothing arguments are standard appear for instance in \cite{reshetnyak2023isothermal}.

From such a locally finite triangulation $H\subset \Sigma$, one can equip each triangular face with the standard Euclidean metric on the triangle with sides of length equal to that of the triangular face.
This gives a complete piecewise Euclidean metric $d_E$ on $\Sigma$, which is singular only at vertices of $H$.
Then for each vertex $v$ of $H$, we can choose some $0<\epsilon_v < 1$ so that the closed disk $E_v$ around $v$ of radius $\epsilon_v$ is contained in the Euclidean triangles incident to $v$ and so that $\epsilon_v$ is less than half the length of any edge incident to $v$. This ensures that the disks are pairwise disjoint.
Then, each such disk can be replaced with a Euclidean disk $D_v$ whose boundary is of equal length to that of $E_v$ and is attached isometrically along the boundary of $E_v$.
Finally, each of these boundaries can be smoothed to obtain a quasi-isometric Riemannian metric $d_\Sigma$.
For any given $\epsilon >0$, we can make this a $(1,\epsilon)$-quasi-isometry by choosing each $\epsilon_v$ to be sufficiently small.





We require some more graph theoretic notation.
We denote the vertex set of a graph $G$ by $V(G)$.
For a vertex $u$ of a graph $G$, we let $N_G(u)$ denote the \emph{neighbourhood} of $u$ in $G$ (the vertices of $G$ adjacent to $u$).
For a set of vertices $A$ of a graph $G$ and a positive real $t$, we let $N^t_G[A]$ be the vertices at distance at most $t$ from $A$ in $G$.
If $A=\{u\}$, then we simply use $N^t_G[u]$.
We denote the vertex set of a graph $G$ by $V(G)$.

To construct the Riemannian surfaces as in \cref{Rie:main2}, we require some finite graphs with useful properties that will embed naturally into our constructed surface.
The \emph{girth} of a graph is equal to the length of its shortest cycle.
A graph is \emph{$k$-regular} if every vertex of $G$ has degree~$k$.
We require finite 4-regular graphs of large girth as constructed (probabilistically) by Erd{\H{o}}s and Sachs~\cite{erdos1963regulare}. For explicit constructions, see \cite{lubotzky1988ramanujan,margulis1982explicit}.

\begin{theorem}[Erd{\H{o}}s, Sachs~\cite{erdos1963regulare}]\label{rie:lem:4-regular}
    There exists finite 4-regular graphs with arbitrarily large girth.
\end{theorem}


We are now ready to prove a version of \cref{Rie:main2} with fixed constants in the considered quasi-isometry.
In this case, we actually construct compact bounded genus surfaces as will be discussed further in \cref{sec:conclude}.

\begin{theorem}\label{Rie:main3}
    For every triple of non-negative reals $K,M,A$ with $M\ge 1$, there exists a compact Riemannian surface $\Sigma$ without boundary, of bounded genus, with systole at least $K$, and with no embedded graph $G\subset \Sigma$ such that $G^{(1)} \hookrightarrow \Sigma$ is a $(M,A)$-quasi-isometry.
\end{theorem}

\begin{proof}
    Fix $K\ge M\ge A \ge 1$, and set $\epsilon = 1/(33M^2)$, and $g=3400M^5K$.

    Let $F$ be some finite connected 4-regular (simplicial) graph of girth at least $g$ as given by \cref{rie:lem:4-regular}.
    Since $F$ is 4-regular, it is Eulerian, and is therefore the union of some collection of edge-disjoint cycles $C_1, \ldots , C_k$.
    Note that every vertex $v$ of $F$ is contained in exactly two of these cycles, say $C_v$ and $C_v'$.
    Now, take a 2-cell embedding of $F$ into some smooth surface $\Sigma$ such that for every vertex $v$ of $F$, the two cycles $C_v$ and $C_v'$ cross at $v$ in the embedding (note that such a surface and embedding can be found by first fixing such a clockwise ordering of incident edges for each vertex, and then adding in 2-cells to create the faces).
    Note that $\Sigma$ has bounded genus and is compact since $F$ is a finite graph with a 2-cell embedding into $\Sigma$.
    We remark that $\Sigma$ is also orientable.
    We still need to choose a Riemannian metric $d$ on $\Sigma$.

    \begin{figure}
    \centering
    \includegraphics[width=0.52\linewidth]{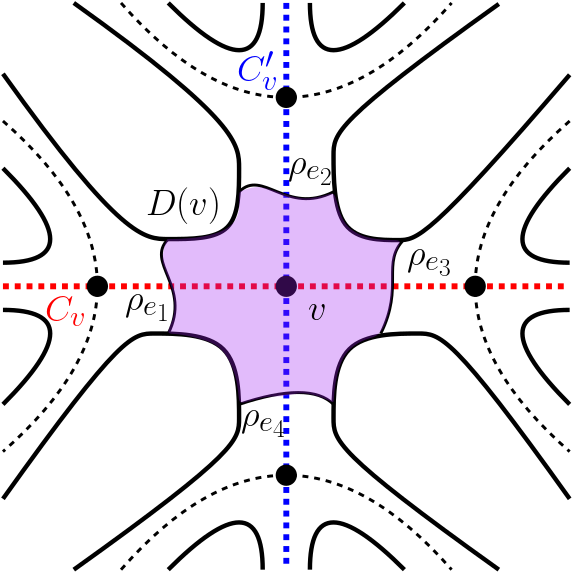}
    \caption{The embedded graph $F$ consists of the dashed edges and the boundary of $D$ is the thick black edges around $F$.
    For the centre vertex $v$, we illustrate the two cycles $C_v$ and $C_v'$ of the Eulerian decomposition that cross at $v$ in red and blue. Also featured are the curves $\rho_{e_1}, \rho_{e_2}, \rho_{e_3}, \rho_{e_4}$ crossing their corresponding edges incident to $v$. Together with part of the boundary of $D$, they bound the subset $D(v)$ of $D$, which is purple in the figure.
    }
    \label{fig1}
\end{figure}

    Now we take some small neighbourhood $D\subset \Sigma$ around $F$ in the embedding so that each face of $F$ in the embedding contains one of the boundary components of $D$.
    For each edge $e$ of $F$, we draw a curve $\rho_e$ in $\Sigma$ between boundary points of $D$ near the midpoint of $e$ and crossing $e$ as in \cref{fig1}.
    For each vertex $v$ of $F$, we let $D(v)$ be the subset of $D$ containing $v$ and bounded by a subset of the boundary of $D$ and $\rho_{e_1} \cup \rho_{e_2} \cup \rho_{e_3} \cup \rho_{e_4}$, where $e_1,e_2,e_3,e_4$ are the four edges of $F$ incident to $F$ (see \cref{fig1} for an illustration).
    More generally, for $X\subseteq V(F)$, we let $D(X)=\bigcup_{x\in X} D(x)$.

    Choose some Riemannian metric $d$ on $\Sigma$ such that 
    \begin{itemize}
        \item for each edge $uv$ of $F$, its ends $u$ and $v$ in the embedding are at distance between $\epsilon / 2$ and $\epsilon$ in $\Sigma$,
        \item the distance between $F$ (in its embedding in $\Sigma$) and the boundary of $D$ is greater than both $K$ and $12M^3$, and
        \item for every $v\in V(F)$ and $r\ge 0$, the points of $D$ at distance at most $r$ from $v$ in $\Sigma$ are contained in $D(N_F^{1+r/\epsilon } [v])$.
    \end{itemize}
    Such a Riemannian metric $d$ on $\Sigma$ can be chosen by starting with a suitable triangulation of $\Sigma$ by a metric $H$ such that $F\subset H \subset \Sigma$ (where the edges of $F$ have length slightly less than $\epsilon$ in $H$) and smoothing the resulting piecewise Euclidean metric as discussed.
    Clearly $\Sigma$ has systole at least $K$ with this choice of Riemannian metric since $F$ is a 2-cell embedding of $\Sigma$ and $F$ has girth at least $3400M^5K\ge 100M^3 K / \epsilon$. 
    
    Suppose now for the sake of contradiction that there exists an embedded graph $G\subset \Sigma$ such that $G^{(1)} \hookrightarrow \Sigma$ is a $(M,A)$-quasi-isometry.
    So, every point of $\Sigma$ is at distance at most $A$ from a point of $G$ in the embedding, the points on every edge are within distance at most $M+A\le 2M$ from each of the edges ends in $\Sigma$, and for any vertices $x,y\in V(G^{(1)})\subset \Sigma$, we have that
    \[
    \frac{1}{M}d(x, y) - A \le d_{G^{(1)}}(x,y) \le Md(x,y) + A.
    \]
    Note furthermore that every point of $\Sigma$ is at distance at most $2M+A\le 3M$ from a vertex of $G^{(1)}$ in the embedding.

    \begin{figure}
    \centering
    \includegraphics[width=0.74\linewidth]{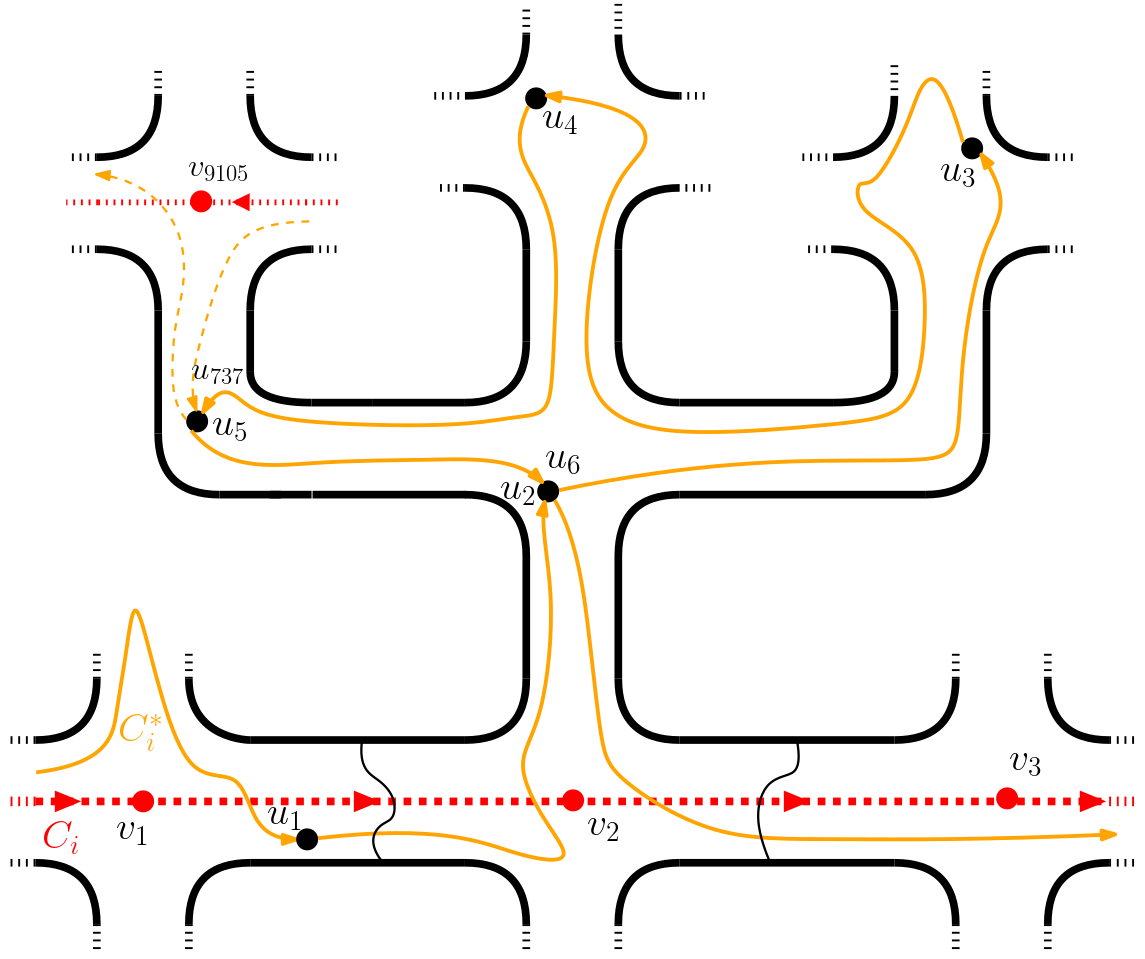}
    \caption{An illustration of (two parts of) the of closed walk $C_i^*$ (orange) of $G^{(1)}$ and $p_{C_i}$ given the cycle $C_i$ (red) of the Eulerian decomposition of $F$.
    We have that $p_{C_i}(u_1)=v_1$, $p_{C_i}(u_2)= p_{C_i}(u_3) = p_{C_i}(u_4) = p_{C_i}(u_5) = p_{C_i}(u_6) = v_2$, and $p_{C_i}(u_{737})=v_{9105}$.
    }
    \label{fig2}
\end{figure}

    Consider one of the cycles $C_i$ in the Eulerian decomposition $C_1, \ldots, C_k$ of $F$.
    We now aim to find a closed walk $C_i^*$ (it needs not be a cycle) of $G^{(1)}$ that closely follows $C_i$ in $\Sigma$ and that is far shorter than $C_i$.
    The reader may wish to refer to \cref{fig2}.
    Choose vertices $c_{i,1}, \ldots , c_{i,\ell_i}$ in order on $C_i$ so that $\frac{1}{2} \le d(c_{i,j},c_{i,j+1}) \le 1$ for each $1\le j \le \ell_i$ (taking $c_{i,\ell_i+1}=c_{i,1}$).
    Note that this can be done since $C_i$ has length at least $g=3400M^5K\ge 2 + 1/\epsilon$.
    As the ends of edge of $C_i$ are at distance at most $\epsilon$ in $\Sigma$, we further have that $\ell_i \le  2|C_i|\epsilon$.
    Now, for each $1\le j \le \ell_i$, let $c_{i.j}^*$ be a vertex of $G^{(1)}$ at distance at most $M+2A\le 3M$ from $c_{i,j}$ in $\Sigma$.
    For each $1\le j \le \ell_i$, we have that $d(c_{i,j}^*,c_{i,j+1}^*) \le d(c_{i,j}^*,c_{i,j}) + d(c_{i,j},c_{i,j+1}) + d(c_{i,j+1},c_{i,j+1}^*) \le  3M+1+3M\le 7M$.
    Therefore, $G^{(1)}$ contains a path $P_{i,j}$ between $c_{i,j}^*$ and $c_{i,j+1}^*$ of length at most $7M^2+A \le 8M^2$.
    Every point of $P_{i,j}$ in $\Sigma$ is at distance at most $8M^3+A\le 9M^3$ from $c_{i,j}^*$ in $\Sigma$, and therefore at distance at most $9M^3 + 3M \le 12M^3$ from $c_{i,j}$ in $\Sigma$.
    Note that this further implies that every such point is contained in $D(N^{400M^5}_{F}[c_{i,j}])$ (and in particular, contained in $D$) since $1 + 12M^3 / \epsilon \le 1+396M^5 \le  400M^5$.
    Let $C_i^*$ be the closed walk in $G^{(1)}$ obtained by concatenating the walks given by the paths $P_{i,1},\ldots , P_{i,\ell_i}$ in order.
    Then, $C_i^*$ is contained in $D$, and furthermore, the length of the closed walk $C_i^*$ is at most $8M^2 (2|V(C_i)|\epsilon) = 16M^2 |V(C_i)|\epsilon < \frac{1}{2}|V(C_i)|$, as $\epsilon = 1/(33M^2)$.

    We further examine $C_i$ and $C_i^*$.
    Each vertex $u$ along the closed walk $C_i^*$ belongs to a subwalk given by some $P_{i,j}$ and is distinct from its end vertex.
    For each such $u$ along the closed walk $C_i^*$, let $c_{i,u}$ be the vertex $c_{i,j}$ of $C_i$ such that $c_{i,j}^*$ is the corresponding starting vertex of $P_{i,j}$.
    Since $u$ is contained in $D(N^{400M^5}_{F}[c_{i,u}])$, it follows that $u$ is contained in $D(N^{800M^5}_{F- N_{C_i}(v)}[v])$ 
    for some $v$ at distance at most $840M^5$ from $c_{i,u}$ in $C_i$.
    Such a vertex $v$ is unique since the induced subgraph of $F$ with vertex set $N^{1640M^5}_F[c_{i,u}]$ is a tree due to the girth of $F$ being at least $3400M^5$.
    (We point out that although $800M^5$ and $840M^5$ here could clearly be improved to $400M^5$, it is this uniqueness why we relax things to $800M^5$ and $840M^5$).
    We define $p_{C_i}(u)=v$ as above for each such vertex $u$ along the closed walk $C_i^*$ of $G^{(1)}$ (note that the closed walk might possibly repeat vertices with $u=u'$ where $u'$ comes after $u$ in the closed walk, however $p_{C_i}(u)$ and $p_{C_i}(u')$ can still be distinct by slight abuse of notation).

    For an illustrative example of the function $p_{C_i}$, see \cref{fig2}.
    One can intuitively think of $p_{C_i}$ as being a function mapping vertices of the closed walk $C^*_i$ to their nearest vertex of $C_i$ in $\Sigma$ that is also close the part of $C_i$ that the portion of the closed walk $C_i^*$ is supposed to be following (we must take care here since there might be vertices much further along $C_i$ that are closer in $\Sigma$ to the given vertex of the closed walk $C_i^*$ than the vertices of $C_i$ that this portion of the closed walk is supposed to be following).


    \begin{figure}
    \centering
    \includegraphics[width=0.60\linewidth]{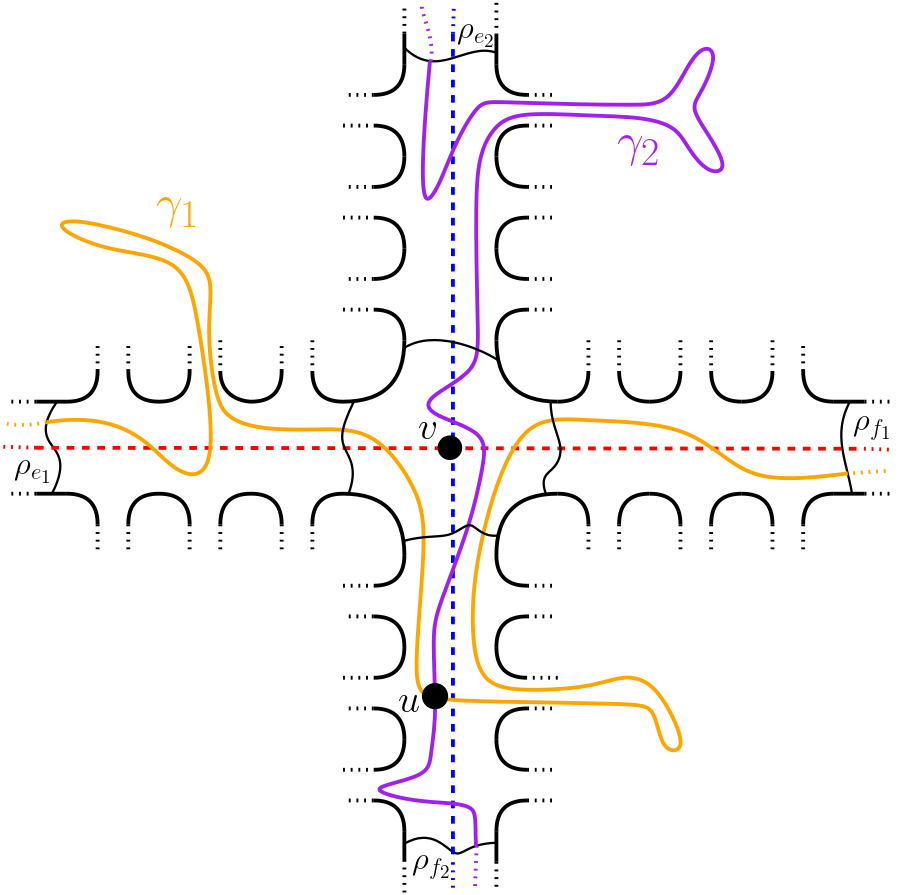}
    \caption{An illustration of $\gamma_1$ (orange) and $\gamma_2$ (purple) contained within $D^*$.
    They intersect at a point $z$ coinciding with a vertex $u$ of the closed walk $C^*_1$ with $p_{C_1}(u)=v$.
    }
    \label{fig3}
\end{figure}

    For each $C_i$, we have that the length of the closed walk $C^*_i$ is strictly less than $\frac{1}{2}|V(C)|$, and so it follows by
    pigeonhole principle that there exists some vertex $v$ of $F$ such that $p_{C_i}(u)\not= v$ for both cycles $C_i$ of the Eulerian decomposition that contains $v$ and every $u$ along the closed walk $C_i^*$.
    Without loss of generality, we may assume that the two cycles of the Eulerian decomposition containing $v$ are $C_1$ and $C_2$.
    To obtain our desired contradiction, we will show that since $C_1^*$ and $C_2^*$ can only cross in $\Sigma$ by sharing common vertices or edges, there must in fact be some vertex $u$ along the closed walk $C_i^*$ for some $i\in \{1,2\}$ such that $p_{C_i}(u)=v$.
    The vertex $u$ will be a vertex of both $C_1^*$ and $C_2^*$ (see \cref{fig3}).

    For every $1\le j \le \ell_1$, since the distance between $c_{1.j}$ and $c_{1,j+1}$ along $C_1$ in $\Sigma$ is at most 1, it follows that the distance between $c_{1.j}$ and $c_{1,j+1}$ in the graph $C_1$ (and also in $G^{(1)}$) is at most $1/\epsilon = 33M^2$.
    So, we can choose vertices $c_{1,s_1}$, $c_{1,t_1}$ at distance between $400M^5 + (400M^5 +1) = 800M^5 +1$ and $800M^5 + 1 + 33M^2 \le 840M^5$ from $v$ in $C_1$ and so that the distance between $c_{1,s_1}$ and $c_{1,t_1}$ in $C_1$ is at least $1600M^5$.
    Since $F$ has girth at least $3400M^5$, we may assume without loss of generality that $s_1=1$ and $t_1\le 1700M^5$.
    Let $Q_1$ be the subpath of $C_1$ between $c_{1,1}$ and $c_{1,t_1}$ that contains $v$.
    Let $P_1^*$ be the subwalk of $C^*_1$ obtained by concatenating the walks given by the paths $P_{1,1},\ldots , P_{1,t_1-1}$ in order.
    So, $P_1^*$ is a walk between $c_{1,1}^*$ and $c_{1,t_1}^*$ in $G^{(1)}$,
    and is also contained in $D(N_F^{400M^5}[V(Q_1)])$.
    Let $e_1$ be the edge of $Q_1$ between $v$ and $c_{1,1}$ and at distance $\lfloor 400M^5 \rfloor$ from $v$ in $C_1$, and similarly, let $f_1$ be the edge of $Q_1$ between $v$ and $c_{1,t_1}$ and at distance $\lfloor 400M^5 \rfloor$ from $v$ in $C_1$.
    Let $Q_1'$ be the subpath of $Q_1$ between $e_1$ and $f_1$.
    As $c_{1,1}^*$ is contained in $D(N_F^{400M^5}[c_{1,1}])$ and $c_{1,t_1}^*$ is contained in $D(N_F^{400M^5}[c_{1,t_1}])$, it follows that the embedding of $P_1^*$ in $\Sigma$ contains a curve $\gamma_1$ contained in $D(N_{F-\{e_1,f_1\}}^{400M^5}[Q_1'])$ that starts on $\rho_{e_1}$ (which is a curve contained in the boundary of $D(N_{F-\{e_1,f_1\}}^{400M^5}[Q_1'])$) and ends on $\rho_{f_1}$ (which again is a curve contained in the boundary of $D(N_{F-\{e_1,f_1\}}^{400M^5}[Q_1'])$).
    Similarly for $C_2$, we find a curve $\gamma_2$ contained in $D(N_{F-\{e_2,f_2\}}^{400M^5}[Q_2'])$ that starts on $\rho_{e_2}$ and ends on $\rho_{f_2}$.
    For an illustration of $\gamma_1$ and $\gamma_2$, see \cref{fig3}.

    Let $D^*= D(N_{F-\{e_1,f_1,e_2,f_2\}}^{800M^5}[v])$.
    Then since $F$ has girth greater than $1600M^5+1$, $D^*$ is homeomorphic to a disk and furthermore $\rho_{e_1},\rho_{f_1}, \rho_{e_2},\rho_{f_2}$ are disjoint curves contained in the boundary of $D^*$ (appearing in order) and for $i\in \{1,2\}$, $\gamma_i$ is a curve contained in $D^*$ between $\rho_{e_i}$ and $\rho_{f_i}$.
    By the Jordan curve theorem, $\gamma_1$ and $\gamma_2$ intersect at some point $z$, which is a common vertex along the walks $P_1^*$ and $P_2^*$.
    Without loss of generality, since $D^* \subseteq D(N^{800M^5}_{F-N_{C_1(v)}}[v]) \cup D(N^{800M^5}_{F-N_{C_1(v)}}[v])$ we may assume that $z \in D(N^{800M^5}_{F-N_{C_1(v)}}[v])$.
    Let $u$ be a vertex along the subwalk $P^*_1$ of $C_1^*$ that coincides with the point $z$ in its embedding in $\Sigma$ as in \cref{fig3}.
    Then $c_{1,u}\in V(Q_1)$, so the distance between $c_{1,u}$ and $v$ in $C_1$ is at most $840M^5$.
    Therefore $p_{C_1}(u) = v$, a contradiction.
\end{proof}

\cref{Rie:main2} now follows from \cref{Rie:main3} by for each pair of positive integers $M, A$ taking a Riemannian surface $\Sigma_{M,A}$ as in \cref{Rie:main3} of systole at least $K$ and joining them together with thick enough tubes (which can be attached somewhere far from $D$ as in the proof of \cref{Rie:main2}, so that this part of the resulting complete Riemannian surface still has no such $(M,A)$-quasi-isometry) and so that the systole remains at least $K$.
This can be done as before by taking a suitable triangulation of the resulting surface by a locally finite metric graph $H$ and adjusting the resulting piecewise Euclidean metric to a $(1,\epsilon)$-quasi-isometric complete Riemannian metric.

\section{Concluding remarks}\label{sec:conclude}

In \cite{davies2025string} we proved that complete Riemannian surfaces $\Sigma$ of genus at most $g$ admit triangulations $G\subset \Sigma$ such that $G^{(1)} \hookrightarrow \Sigma$ is a $( 10^6 , O(g) )$-quasi-isometry.
\cref{Rie:main3} shows that the dependence on $g$ is necessary.

One can obtain explicit bounds of the genus of the Riemannian surface constructed in \cref{Rie:main3}.
Taking $K=M=A\ge 1$ for simplicity, in the proof of \cref{Rie:main3} we start with a 4-regular graph $F$ of girth at least $3400M^6$ and create a surface in which $F$ has a 2-cell embedding. By Euler's formula it can easily be shown that the genus of the resulting surface is at most $|V(F)|$.
Since there are $4$-regular graphs with girth at least $g$ and at most $3^g$ vertices~\cite{lazebnik1997upper}, it follows that the surface can be taken to have genus at most $3^{3400M^6}$.

With more care, for fixed $K$ $M$, it can be shown that the resulting Riemannian surface has genus at most $3^{O(A^2)}$.
Thus, there are complete Riemannian surfaces $\Sigma$ of genus at most $g$ admitting no triangulation $G\subset \Sigma$ such that $G^{(1)} \hookrightarrow \Sigma$ is a $(10^6 , \Omega(\sqrt{\log g })$-quasi-isometry.
It would be interesting to obtain better upper and lower bounds for the constants of such a quasi-isometry depending on the genus $g$.
It also remains open whether such $(1,f(g))$-quasi-isometric simplicial triangulations exist.

\section*{Acknowledgements}

We thank 
Matija Bucić, 
Agelos Georgakopoulos, and
Federico Vigolo 
for helpful discussions and comments.

\bibliographystyle{amsplain}
{\small
\bibliography{QI}}

\end{document}